\documentclass[11pt]{amsart}

\usepackage{enumerate}
\usepackage{amssymb}
\usepackage{amsmath,amscd}
\usepackage[colorinlistoftodos]{todonotes}
\usepackage{amsthm}

\newcommand{\tr}{\operatorname{tr}}

\newcommand{\N}{\mathbb{N}}
\newcommand{\R}{\mathbb{R}}

\newcommand{\C}{\mathbb{C}}

\newcommand{\HH}{\mathbb{H}}
\newcommand{\Oct}{\mathbb{O}}
\newcommand{\F}{\mathbb{F}}

\newtheorem{theorem}{Theorem}

\newtheorem{lemma}[theorem]{Lemma}

\newtheorem{maintheorem}{Theorem}

\theoremstyle{definition}

\theoremstyle{remark}
\newtheorem{remark}[theorem]{Remark}
\newtheorem{example}[theorem]{Example}
\newtheorem{question}[theorem]{Question}

\title{Diameter and focal radius of submanifolds}

\author[R.~Mendes]{Ricardo A. E. Mendes}
\address{University of Oklahoma, USA}
\email{ricardo.mendes@ou.edu}


\begin{document}

\begin{abstract}
In this note, we give a characterization of immersed submanifolds of simply-connected space forms for which the quotient of the extrinsic diameter by the focal radius achieves the minimum possible value of $2$. They are essentially round spheres, or the ``Veronese'' embeddings of projective spaces. The proof combines the classification of submanifolds with planar geodesics due to K. Sakamoto with a version of A. Schur's Bow Lemma for space curves. Open problems and the relation to recent work by M. Gromov and A. Petrunin are discussed.
\end{abstract}
 \maketitle 


\section{Introduction}

Denote by $M^n(\kappa)$ the complete simply-connected $n$-dimensional Riemannian manifold with constant curvature $\kappa$. For simplicity we will assume  $\kappa=0,1,-1$, so that  $M^n(\kappa)$ is the Euclidean space $\R^n$, the unit sphere $S^n$, or the hyperbolic $n$-space, respectively.

Let $M$ be a closed (that is, compact without boundary), smooth, connected manifold, immersed in $M^n(\kappa)$. The \emph{extrinsic diameter} of $M$ is defined as the maximum distance, measured in $M^n(\kappa)$, of any pair $x,y\in M$.  The \emph{focal radius} of $M$ is defined as the supremum of $r>0$ such that the normal exponential map $\nu^r(M)\to M^n(\kappa)$ is a local diffeomorphism when restricted to the set $\nu^r(M)$ of normal vectors of length less than~$r$.

The goal of this note is to prove the following:
\begin{maintheorem}
\label{MT}
Let $M$ be a closed, smooth, connected, immersed submanifold of $M^n(\kappa)$. Then the extrinsic diameter of $M$ is at least $2$ times its focal radius.

Moreover, equality holds if and only if: $M$ is a round sphere, and the immersion is a totally umbilical embedding; or $M$ is a (real, complex, quaternionic, or octonionic) projective space, and the immersion is a ``Veronese'' embedding; or $M$ is a round sphere, and the immersion is the composition of the covering map to the round real projective space with a Veronese embedding.
\end{maintheorem}

We remind the reader that totally umbilical spheres in $\R^n$ are distance spheres in affine subspaces. In $S^n$ they are intersections of $S^n\subset \R^{n+1}$ with affine subspaces of $\R^{n+1}$, and similarly for hyperbolic space, via its natural embedding into Minkowski space as a hyperboloid, see \cite[page 28]{Sakamoto77} for more details. Veronese embeddings of (real, complex, quaternionic, or octonionic) projective spaces into spheres are defined using certain spaces of matrices. More generally, one obtains embeddings into $M^n(\kappa)$ by composing these with totally umbilical embeddings of spheres into  $M^n(\kappa)$, see Section \ref{SS:Veronese} for more details.

Theorem \ref{MT} is related to some recent results in \cite{Gromov22, Gromov22designs, Petrunin24tori, Petrunin24Veronese}. See Section \ref{S:open} for more details, and a discussion of open problems.

\subsection*{Acknowledgements}
I would like to thank B. Schmidt, who, in conversations inspired by \cite{IST22}, initially suggested this project. I would also like to thank L. Ni, M. Ghomi, and A. Petrunin for help at later stages.

%
%

\section{Preliminaries}
\subsection{Bow Lemma}


For a unit-speed smooth curve $\gamma\colon\R\to M^n(\kappa)$ (or any Riemannian manifold), the \emph{(geodesic) curvature} at $\gamma(t)$ is defined as $\|\gamma''(t)\|$.

A (geodesic) circle, that is, the set of points  in $M^2(\kappa)$ at a certain constant distance from a point, is a curve of constant geodesic curvature. Composing such a curve with a totally geodesic embedding of $M^2(\kappa)$ in $M^n(\kappa)$ yields what we will refer to as \emph{planar circles}.

We will need the following  consequence of \cite[Theorem 1.1]{AB96}:
\begin{lemma}\label{L:Bow}
Let  $\eta\colon [0,\ell/2]\to M^2(\kappa)$ be a constant-curvature unit-speed parametrization of a half-circle of radius $r$.  Let $\gamma\colon [0,\ell/2]\to M^n(\kappa)$ be a smooth unit-speed curve with geodesic curvature at most that of $\eta$. Then the chord length of $\gamma$ is no smaller than the chord length of $\eta$, that is: 
\[d_{M^n(\kappa)}(\gamma(0),\gamma(\ell/2)) \geq 2r. \] 

Moreover, if equality holds, then $\gamma$ is contained in a planar circle.
\end{lemma}
\begin{proof}
Let $c$ denote the curvature of $\eta$. In the case $\kappa=1$, note that $\ell/2=\pi\sin(r)$, so that the length of $\gamma$ plus its chord length is at most $2\pi$, with equality only when $\gamma$ is a geodesic between a pair of antipodal points in $M^n(\kappa)=S^n$, in which case the conclusions of the Lemma are clearly satisfied. Thus we may assume that the length of $\gamma$ plus its chord length is strictly less than $2\pi$. This means that, for any $\kappa$, we can apply \cite[Theorem 1.1]{AB96}, and conclude that there exists a unit-speed curve $\eta'\colon [0,\ell/2]\to M^2(\kappa)$ with constant curvature $c'\leq c$, and the same chord length as $\gamma$. By elementary geometry  in $M^2(\kappa)$, for minor arcs of a fixed length, the chord length is strictly decreasing as a function of the geodesic curvature (see, for example, \cite[Appendix]{AB96}). Thus $d_{M^n(\kappa)}(\gamma(0),\gamma(\ell/2)) \geq 2r$, and, if equality holds, we must have $c'=c$. Thus the last statement follows from the rigidity statement in  \cite[Theorem 1.1]{AB96}.
\end{proof}
\begin{remark}
In the statement of Lemma \ref{L:Bow}, if one allows the comparison curve $\eta$ to have variable curvature, any length, but requires it to be chord-convex, one arrives at a statement known as the ``Bow Lemma''. In Euclidean space ($\kappa=0$), this is a classical result due to A. Schur, see \cite[page 46]{Chern67}, or \cite[Theorem A, page 31]{HopfLNM1000}\footnote{In the given references, Schur's theorem is only stated for curves in $\R^3$, but, as pointed out in \cite[page 101]{Ni23}, the proof given in \cite{HopfLNM1000} generalizes to curves in $\R^l$ for any $l\geq 2$.}. In the sphere ($\kappa=1$), this was recently obtained in \cite{Ni23}. For a version for CAT$(\kappa)$-spaces, for $\kappa\leq 0$, (but without the rigidity statement), see \cite{Ghomi23}.
\end{remark}

\subsection{Focal radius and normal curvatures} \label{S:focalradius}
Let $M$ be an immersed submanifold of $M^n(\kappa)$, let $p\in M$, and $u\in T_pM$ be a unit vector. The \emph{normal curvature} of $M$ at $p$ in the direction of $u$ is defined as the geodesic curvature at $t=0$ of the unit-speed geodesic starting at $p$ with initial velocity $u$, or, equivalently, $\| I\!I(u,u)\|$, where $I\!I$ denotes the second fundamental form of $M$ in $M^n(\kappa)$. 

Given a unit normal vector $\xi\in\nu_p M$, the \emph{focal distances} in the direction of $\xi$ are the values of $t\in\R$ such that the differential of the normal exponential map at $t\xi$ has positive nullity. Equivalently, $t$ is a focal distance in the direction of $\xi$ if and only if the shape operator $A_\xi$ has an eigenvalue $\lambda$ of the form $\lambda=1/t, \cot(t),\coth(t)$,  for $\kappa=0,1,-1$, respectively, see e.g. \cite[Proposition 1.1]{CR78}. 

Note that, for fixed $p\in M$, the minimum (positive) focal distance $t_{\min}$ over all unit $\xi\in \nu_p M$ and the maximum normal curvature $c_{\max}$ over all unit $u\in T_p M$ satisfy $c_{\max}=1/t_{\min}, \cot(t_{\min}),\coth(t_{\min})$,  for $\kappa=0,1,-1$, respectively. 

In particular, the \emph{focal radius of $M$}, which is equal to the infimum, over all $p\in M$ and all unit $\xi\in\nu_p M$, of the smallest positive focal distance in the direction of $\xi$, is related by analogous formulas to the supremum, over all $p\in M$ and all unit $u\in T_p M$, of the normal curvature.

\subsection{Veronese embeddings of projective spaces}
\label{SS:Veronese}
We will use notations from \cite{Little76}, to which we refer the reader for more details. 

Let $\F$ denote either the reals $\R$, the complex numbers $\C$, the quaternions $\HH$, or the octonions $\Oct$. Let $M^{l+1}(\F)$ denote the real vector space of all $(l+1)\times (l+1)$-matrices with entries in $\F$, where $l\in\N$, except for $\F=\Oct$, for which $l\in\{1,2\}$. The space $M^{l+1}(\F)$ is endowed with the inner product 
\[M_1\cdot M_2=\frac{1}{2}\tr\left(M_1 \overline{M_2^t} + M_2\overline{M_1^t}\right).\]
Let 
\[\F P^l= \{ M\in M^{l+1}(\F) \mid M=\overline{M^t},\ M=M^2,\ \operatorname{rank}(M)=1\}. \]

The subset $\F P^l$ of the Euclidean vector space $M^{l+1}(\F)$ is a connected, closed, smooth, embedded submanifold, which is isometric to the corresponding projective space with the (appropriately rescaled) Fubini--Study metric. 

Geometrically, the points of $\F P^l$ are subspaces of   $\F^{l+1}$ with dimension $1$ over $\F$, and they are identified via the Veronese embedding with the corresponding orthogonal projection matrix. The spaces $\F P^l$ (together with round spheres) also coincide with the class of compact symmetric spaces of rank one.

Note that $\F P^l$ is contained in the sphere of radius $\sqrt{l/(l+1)}$ in the affine subspace of Hermitian matrices with trace $1$, with center $\operatorname{Id}/(l+1)$. 

By rescaling, one thus obtains isometric embeddings of (appropriately rescaled) $\F P^l$ into spheres of any radius. The compositions of these embeddings with isometric embeddings of spheres into $M^n(\kappa)$ as totally geodesic or totally umbilical submanifolds (see \cite[page 28]{Sakamoto77} for more details) will be referred to as the \emph{Veronese embeddings} of $\F P^l$.

\begin{theorem}{\cite[Theorems 1, 2, and their Corollaries]{Sakamoto77}} \label{T:Veronese}
Under a given Veronese embedding, all geodesics of $\F P^l$ are mapped to planar circles of the same length.
\end{theorem}

\section{Proof of Theorem \ref{MT}}
Let $r$ denote the focal radius of $M$. Pick any unit-speed geodesic $\gamma:\R\to M$. Then the geodesic curvature of $\gamma$ in $M^n(\kappa)$ is at 
most that of a circle of radius $r$ in $M^2(\kappa)$ (see Section \ref{S:focalradius}). Let $\ell$ denote the length of this circle. 
It follows from Lemma \ref{L:Bow} that the chord length of $\gamma_{[0,\ell/2]}$ is at least $2r$. In particular, the extrinsic diameter of $M$ is at least $2r$.

Suppose the extrinsic diameter of $M$ is equal to $2r$. Then the rigidity statement in Lemma \ref{L:Bow} implies that, for every geodesic $\gamma$, the set $\gamma([0,\ell/2])$ is contained in a totally geodesic copy of $M^2(\kappa)$. This implies $M$ has planar geodesics. By \cite[Theorem 3]{Sakamoto77}, $M$ is either a totally umbilical sphere, or a Veronese embedding of some $\F P^l$, or the composition of the covering map $S^l\to \R P^l$ with a Veronese embedding of $\R P^l$.

Conversely, assume $M$ is either a totally umbilical sphere, or a Veronese embedding of some $\F P^l$. In either case, every geodesic is a planar circle of the same radius $r$ (see Theorem \ref{T:Veronese}). Thus the focal radius of $M$ is $r$, and, since every two points of $M$ can be joined by a geodesic, the extrinsic diameter of $M$ is $2r$.

\section{Remarks and open questions}\label{S:open}
In this section we consider only the case $\kappa=0$, that is, we consider immersions into Euclidean space $\R^n$.

In view of Theorem \ref{MT}, it is natural to ask:
\begin{question}\label{Q:ed/fr}
Fix a connected, closed, smooth manifold $M$. What is the infimum of the quantity (extrinsic diameter)$/$(focal radius) over all smooth immersions of $M$ into Euclidean space $\R^n$ (for all $n$)? Is the infimum achieved? If so, classify the immersions that achieve the infimum.
\end{question}
Note that Theorem \ref{MT} completely solves Question \ref{Q:ed/fr} when $M$ is diffeomorphic to a compact rank-one symmetric space.

Another natural question is obtained by replacing ``extrinsic diameter'' with ``circumradius'' in Question \ref{Q:ed/fr}. By circumradius we mean the smallest radius  of a closed ball in $\R^n$ that contains the image of $M$ under the immersion. After rescaling and translating, one may assume $M$ has circumradius $1$ and is contained in the closed unit ball $B^n$ centered at the origin. Noting that the focal radius is the reciprocal of the maximum normal curvature (see Section \ref{S:focalradius}), one arrives at:
\begin{question}\label{Q:circumradius}
Fix a connected, closed, smooth manifold $M$. What is the infimum of the maximum normal curvature over all smooth immersions of $M$ into $B^n$ (for all $n$)? Is the infimum achieved? If so, classify the immersions that achieve the infimum.
\end{question}

We first note that any immersion of any closed manifold $M$ into $B^n$ must have maximum normal curvature at least $1$. Moreover, if equality holds, then $M$ must be diffeomorphic to a sphere, embedded as a totally geodesic subsphere of $\partial B^n=S^{n-1}$. Indeed, using the Bow Lemma, all geodesics of $M$ must be mapped to great circles in $S^{n-1}$. 

Question \ref{Q:circumradius} has received much attention recently. 
In particular:
\begin{itemize}
\item \cite[1.1.C]{Gromov22} For every $M$, the infimum in Question  \ref{Q:circumradius} is strictly less than $\sqrt{3}$. 
\item \cite{Gromov22designs,Petrunin24tori} For $M$ diffeomorphic to the torus $T^m$, the infimum in Question  \ref{Q:circumradius} is $\sqrt{3m/(m+2)}$, and it is achieved.
\item \cite{Petrunin24Veronese} For $M$ diffeomorphic to some projective plane $\F P^2$, the infimum in Question  \ref{Q:circumradius} is $2/\sqrt{3}$, achieved precisely by the Veronese embeddings.
\end{itemize}

By Jung's Theorem, for any compact subset of $\R^n$ one has:
\[ \sqrt{2} \text{ circumradius} < \text{extrinsic diameter}\leq 2 \text{ circumradius}. \]
Together with the results mentioned above, one obtains:
\begin{itemize}
\item  For every $M$, the infimum in Question  \ref{Q:ed/fr} is strictly less than $2\sqrt{3}$. 
\item For $M$ diffeomorphic to the torus $T^m$, the infimum in Question \ref{Q:ed/fr} is between $\sqrt{2}\sqrt{3m/(m+2)}$ and $2\sqrt{3m/(m+2)}$. Note that, for $m\geq 5$, the lower bound is strictly larger than $2$.
\end{itemize}

The upper bound for the torus obtained above is not sharp:
\begin{example}
Define $\phi:\R^3\to\R^6$ by $\phi(x,y,z)=$
\[\frac{1}{\sqrt{3}}\left(\cos(\sqrt{3}x),\sin(\sqrt{3}x),\cos(\sqrt{3}y),\sin(\sqrt{3}y),\cos(\sqrt{3}z),\sin(\sqrt{3}z)\right)\]
Note that the image of $\phi$ is the Clifford torus contained in the unit sphere $S^5$. Let $X$ be the plane in $\R^3$ defined by the equation $x+y+z=0$. The restriction of $\phi$ to $X$ induces an isometric immersion of a flat $2$-torus into $\R^6$, all of whose normal curvatures are $\sqrt{3m/(m+2)}=\sqrt{3/2}$, see \cite[page 13]{Gromov22designs}. Solving a simple optimization problem yields that the extrinsic diameter of this $2$-torus is $\sqrt{3}$. 
\end{example}

\bibliography{ref}

\begin{thebibliography}{{Gro}22b}

\bibitem[AB96]{AB96}
Stephanie~B. Alexander and Richard~L. Bishop.
\newblock Comparison theorems for curves of bounded geodesic curvature in
  metric spaces of curvature bounded above.
\newblock {\em Differential Geom. Appl.}, 6(1):67--86, 1996.

\bibitem[Che67]{Chern67}
S.~S. Chern.
\newblock Curves and surfaces in {E}uclidean space.
\newblock In {\em Studies in {G}lobal {G}eometry and {A}nalysis}, pages 16--56.
  Math. Assoc. America,, 1967.

\bibitem[CR78]{CR78}
Thomas~E. Cecil and Patrick~J. Ryan.
\newblock Focal sets of submanifolds.
\newblock {\em Pacific J. Math.}, 78(1):27--39, 1978.

\bibitem[{Gho}23]{Ghomi23}
Mohammad {Ghomi}.
\newblock {Convexity and rigidity of hypersurfaces in Cartan-Hadamard
  manifolds}.
\newblock {\em arXiv e-prints}, page arXiv:2308.15454, August 2023.

\bibitem[{Gro}22a]{Gromov22designs}
Misha {Gromov}.
\newblock {Curvature, Kolmogorov Diameter, Hilbert Rational Designs and
  Overtwisted Immersions}.
\newblock {\em arXiv e-prints}, page arXiv:2210.13256, October 2022.

\bibitem[{Gro}22b]{Gromov22}
Misha {Gromov}.
\newblock {Isometric Immersions with Controlled Curvatures}.
\newblock {\em arXiv e-prints}, page arXiv:2212.06122, December 2022.

\bibitem[Hop89]{HopfLNM1000}
Heinz Hopf.
\newblock {\em Differential geometry in the large}, volume 1000 of {\em Lecture
  Notes in Mathematics}.
\newblock Springer-Verlag, Berlin, second edition, 1989.
\newblock Notes taken by Peter Lax and John W. Gray, With a preface by S. S.
  Chern, With a preface by K. Voss.

\bibitem[IST22]{IST22}
Mark {Iwen}, Benjamin {Schmidt}, and Arman {Tavakoli}.
\newblock {Characterizing unit spheres in Euclidean spaces via reach and
  volume}.
\newblock {\em arXiv e-prints}, page arXiv:2202.06161, February 2022.

\bibitem[Lit76]{Little76}
John~A. Little.
\newblock Manifolds with planar geodesics.
\newblock {\em J. Differential Geometry}, 11(2):265--285, 1976.

\bibitem[Ni23]{Ni23}
Lei Ni.
\newblock A {S}chur's theorem via a monotonicity and the expansion module.
\newblock {\em J. Reine Angew. Math.}, 805:101--114, 2023.

\bibitem[Pet24a]{Petrunin24tori}
Anton Petrunin.
\newblock Gromov's tori are optimal.
\newblock {\em Geom. Funct. Anal.}, 34(1):202--208, 2024.

\bibitem[{Pet}24b]{Petrunin24Veronese}
Anton {Petrunin}.
\newblock {Veronese minimizes normal curvatures}.
\newblock {\em arXiv e-prints}, page arXiv:2408.05909, August 2024.

\bibitem[Sak77]{Sakamoto77}
Kunio Sakamoto.
\newblock Planar geodesic immersions.
\newblock {\em Tohoku Math. J. (2)}, 29(1):25--56, 1977.

\end{thebibliography}
\bibliographystyle{alpha}
\end{document}